\documentclass[12pt,a4]{article}

\newenvironment{demo*}{\vspace{3mm}\noindent{\bf Proof.}}{\hfill $\Box$ \vspace{3mm}}

\usepackage[dvips]{graphicx}
\usepackage{amsmath,amsthm,amssymb,amscd}
\newtheorem{thm}{Theorem}

\numberwithin{equation}{section} \numberwithin{thm}{section}
\numberwithin{lema}{section}

\def\cen{\centerline}

\def\E{\text{E}}
\def\rd{\text{d}}
\def\th{{\theta}}
\def\sig{\sigma}
\def\vv{\varepsilon}
\def\Var{\text{Var}}
\def\del{\delta}
\def\la{\lambda}
\def\nn{\nonumber}
\def\vs{\vskip .12in}
\def\non{\noindent}

\begin{document}

\baselineskip=22pt

    \title{ Empirical Likelihood for Right Censored Lifetime Data }
    \author{ Shuyuan He$^{1*}$,  Wei Liang$^{2*}$, Junshan Shen$^{3*}$
and Grace Yang $^{4}$
    \footnote{Research of authors 1, 2 and 3 were  Supported by  National Natural Science Foundation of
    China(11171230, 10801003). Research of author 4 was supported by the US National Science Foundation, while working at the Foundation.
       }    \\
 Capital Normal University$^1$, 
Xiamen University$^2$, \\ 
Peking University$^3$, 
University of Maryland$^4$
\\}

\maketitle

\cen{{\large {\bf Abstract}}} \

This paper considers the empirical likelihood (EL)  construction of
confidence intervals for a linear functional $\th$ based on right
censored lifetime data. Many of the results in literature show that
$-2\log$(empirical likelihood ratio) has a limiting
scaled-$\chi^2_1$ distribution, where the scale parameter is a
function of the unknown asymptotic variance. The scale parameter has
to be estimated for the construction.  Additional estimation would
reduce the coverage accuracy for $\th$. This diminishes a main
advantage of the EL method for censored data. By utilizing certain
influence functions  in an estimating equation, it is shown that
under very general conditions, $-2\log$(EL ratio) converges weakly
to a standard $\chi^2_1$ distribution and thereby eliminates the
need for estimating the scale parameter. Moreover, a special way of
employing influence functions eases the otherwise very demanding
computations of the EL method. Our approach yields smaller
asymptotic variance of the influence function than those comparable
ones considered by Wang \& Jing (2001) and Qin \& Zhao (2007). Thus
it is not surprising that confidence intervals using influence
functions give a better coverage accuracy as demonstrated by
simulations.

 \vspace{0.5cm}

\noindent {\it AMS} 1991 {\it subject classifications}.
    Primary 62H99; Secondary 62H05.\vspace{0.5cm}

\setcounter{equation}{0}
\renewcommand{\theequation}{1.\arabic{equation}}


\section{Introduction}
\def\theequation{1. \arabic{equation}}
\setcounter{equation}{0}

Let $Y$ be a non-negative random variable with distribution function
$F(y) = P[Y \leq y]$. Let $C$ be a nonnegative random variable with
distribution function $G(x)= P[C \leq x]$ and $C$ is independent of
$Y$. Instead of $Y$, we observe  $Z=\min(Y,C)$ and the indicator
$\delta=I[Y\leq C]$ of the event $[Y \leq C].$
Both $F$ and $G$ are assumed continuous but unknown. In this paper
we use the empirical likelihood method with right-censored data to
study the problem of constructing  confidence intervals for $\th$, a
functional of the distribution $F$ defined by $\E g(Y,\th)=0$. For
example, if $g(x,\th)= \xi(x) -\th$, for some function $\xi$ we get
\begin{equation}
\th= \E\,\xi(Y)=\int_0^\infty \xi(x) \,\rd F(x).  \label{1.1}
\end{equation}

The problem of estimating $\th$  in (1.1) with a sample of $n$
i.i.d. observations of $(Z, \del)$ has been studied by many authors.
If $\xi(Y) = I[Y \leq y]$ for a fixed $y$, then $\th = F(y)$. A
well-known nonparametric maximum likelihood and asymptotically
optimal estimator of $F(y)$ is the Kaplan-Meier (KM) estimator $F_n$
as defined in (2.3).
A natural estimator of $\th$ is
\begin{equation}
\th_n = \int_0^\infty \xi(s) \,\rd F_n(s). \label{1.2}
\end{equation}

For an arbitrary function $\xi$, several authors, e.g. Yang (1994)
have shown that under the condition of finite second moment,
\begin{equation}
\int_0^{\infty}\frac{\xi^2(s)}{\overline G(s)} \,\rd F(s) <\infty,
\label{1.3}
\end{equation}
the asymptotic distribution of $\sqrt{n}(\th_n-\th)$, as $n$ goes to
infinity, is  normal $N(0,\,\sig^2)$,  where
\begin{equation}
\sig^2=\int_0^{\infty} \frac{(\overline F(s)  \xi(s) -
\psi(s))^2}{\overline F^2(s) \overline G(s)} \,\rd F(s), \
\psi(s)=\int_s^\infty \xi(x) \,\rd F(x), s\geq 0, \label{1.4}
\end{equation}
and $\overline G =1-G, \overline F=1-F$.

Confidence intervals for $\th$ can be constructed using the
asymptotic normal distribution $N(0,\,\sig^2)$. Alternatively, the
EL method can be used as to be investigated in this paper. Employing
either method, one needs to deal with a rather complicated form of
the asymptotic variance $\sig^2$. Among other things, it is
computationally demanding.

To use the normal distribution $N(0,\,\sig^2)$, it is necessary to
estimate the unknown variance $\sig^2$. Stute (1996) proposed a
jackknife estimator to replace $\sig^2$ in the calculation. Although
any consistent estimator $\sig_n^2$ of $\sig^2$ can be used, the
convergence rate of $\sig_n^2$ is generally unknown. Substitution by
the estimate $\sig_n^2$ tends to reduce the coverage accuracy for
$\th$ as compare to the case of known $\sig^2$.

The usefulness of the EL method for constructing  confidence
interval/regions has been well established in a wide variety of
situations, see e.g., DiCiccio et al. (1991) and Chen (1994), and an
extensive literature review in Owen (2001) to that day.  Let
$R(\th)$ denote the EL ratio function of a one-dimensional
parameter $\th$ for $n$ i.i.d ``complete'' observations. Owen (1988)
proved that under certain regularity conditions, $-2\log R(\th)$
converges to
a chi-squared distribution with one degree of freedom. 
The EL method gives confidence intervals for $\th$ as $\{\th :\,
-2\log R(\th) \leq c_{1-\alpha} \} $, where $c_{1-\alpha}$ is the
$(1-\alpha)$th quantile of the $\chi_1^{2}$ distribution. Here the
construction of confidence intervals does not require estimation of
asymptotic variance. In view of a complicated variance formula in
(1.4), this would have provided a welcome method for censored data.
However, as far as we know, for censored data, the asymptotic
standard chi-squared distribution holds only in some special cases
see, e.g. Owen (Chapter 6, 2001). More recent literature shows that
most of the asymptotic distributions involve weights which are
functions of unknown variances or covariance matrices. This is the
case, for example, in the following papers.  Li and Wang (2003)
studied right-censored regression models, Ren (2008) used weighted
EL under a variety of censoring models,  Wang \& Jing (2001) and Qin
\& Zhao (2007) estimated  functionals $\th$, and  Hjort\& McKeague
\& van Keilegom, in their extension of the scope of the EL method
(2009), obtained an asymptotic distribution (Theorem 2.1) which is a
sum of weighted chi-squared distributions with unknown weights.
Therefore, using these results to construct confidence intervals for
$\th$ still  require an additional estimation of the unknown
$\sig^2$. This diminishes a main advantage of the EL method for
censored data.

The EL ratio $R(\th)$ is obtained by utilizing auxiliary information
on $\th$ through a set of estimating equations. In this paper, we
show that by using certain influence functions with a special
construction of estimating equations in the EL ratio, the asymptotic
distribution of  -2$\log R(\th)$ of the functional $\th$ in (1.1) is
a standard $\chi_1^2$ without involving any unknown scale parameter.
Our approach transfers the problem of estimating $\sig^2$ to the
influence functions. As a result, it also significantly simplifies
the often intensive computations of the EL method for censored data.

Our work is motivated by the work of Wang \& Jing (2001) and Qin \&
Zhao (2007). Wang \& Jing (2001) obtained an EL ratio by first
finding an estimating equation for a certain complete sample and
then modifying the estimating equation for the right censored
sample. The resulting estimating equation is a sum instead of a
product (inherent of the product limit estimator). With this
approach, Wang \& Jing (2001) use the estimating function
 $ M_1(Z,
\delta, \th)$ for  $\th$ (see (3.7) ) and Qin \& Zhao (2007)  use $
M_2(Z, \delta, \th)$ (see (3.8)) for estimating the mean residual
life $\E (Y-t_0|Y\geq t_0)$ at age $t_0.$  However, both of these
papers obtain an asymptotic scaled $\chi^2_1$ distribution for
$-2\log R(\th).$




Instead of $M_1$ and $M_2$, we use influence functions. We compute
the influence functions $W(Z, \delta, \theta)$ of
$$\mu_n=\int_0^\infty g(x,\th)\,\rd F_n(x),$$
as defined by (3.6),
where $F_n$ is the Kaplan-Meier estimator. The influence functions
$W's$ are to be utilized to construct an estimation function for the
EL method.  Numerous examples of the function $g$ are given in
Section 2.




The paper is organized as follows. Preliminaries assumptions and
examples of $\th$ and $g$ are given in Section 2. The influence
function $W(Z,\del, \th)$ are given in Section 3.  It is shown in
Theorem 3.1 that asymptotically $\sqrt{n}(\mu_n-\mu)$ is a partial
sum of $n$ independent influence functions $W(Z_j,\del_j, \th)$ (an
IID representation), or is asymptotically linear. Here $\mu$ denotes
$E\,g(Y,\th)$ and the
 the condition $\E\,g(X,\th) = 0$ is not imposed in Section 3.  These results are general for any $\xi(x)$
having finite second moment (1.3)
 and  no restrictions are placed on the upper boundaries of $X$ and $C$.

An IID representation of the Kaplan-Meier estimator has been
obtained by many authors e.g., Lo \& Singh (1986), Stute \& Wang
(1993) and Chen \& Lo (1997) using different approaches, under
different conditions and in different forms. See Yang (1997) and
references therein. Here, we use an influence function $W(Z,\del,
\th)$ of $\th_n$ in (1.2)
obtained in He \& Huang(2003). We show that the variance of  the
influence function is smaller than that of $M_1$ and $M_2$
(see Remark 3.1 in Section 3), which results in an improvement  of
the asymptotic coverage accuracies of $\th$. In Section 4,
estimation of the influence functions is carried out in Theorem 4.1.
The weak convergence of $-2\log R(\th)$ to the  standard $ \chi_1^2$
distribution without any scale parameter is proved in Theorem 4.2
which justifies the EL construction of confidence intervals for
censored data. In Section 5,
 simulation comparison of the new method with that of the
 scaled $\chi_1^2$ distribution is presented.  The amount of improvement depends on the form of $\th$. For survival function or mean, the coverage ratios computing from the traditional normal
approximation and the EL method are about the same. The EL method
performs better for more complicated $\th$.  Most of the proofs are
relegated to the Appendix.

\section{Preliminaries, Assumptions and Examples}
\def\theequation{2. \arabic{equation}}
\setcounter{equation}{0}

For any right continuous monotone function $h({x})$, let $h({x-})$
or $h_{-}(x)$ denote the left continuous version of $h({x})$ and the
curly brackets $h\{{x}\}$ denote the difference $h({x})-h({x-})$.
Then $h\{x\}=\rd h(x)$. For any cumulative distribution function
$F$, let $\overline F=1-F$. \begin{normalsize} {\Large •}
\end{normalsize}Assume that
\begin{equation}
F(x)=P(Y\leq x),\quad G(x) = P(C \leq x)  \ \ \text{and} \ \
H(x)=P(Z\leq x) \label{2.1}
\end{equation}
are continuous cumulative distributions of $Y$, $C$ and
$Z=\min(Y,C)$, respectively. Let $[0, b_H]$ be the range of $H$,
where
$$b_H =\sup\{x : H(x) < 1\}.$$
$b_F$ and $b_G$ are similarly defined for $F$ and $G$. Then
$b_H=\min( b_F, b_G)$.

Given a sample of $n$ i.i.d. random vectors $(Z_i, \del_i)$,
$i=1,2,\cdots,n$, of $(Z, \del)$, their empirical distribution
functions are given by:
\begin{eqnarray}
&&H_n^1(x) = \frac 1n \displaystyle\sum^n_{j=1} I[Z_j \leq x,\delta_j=1], \notag \\
&&H_n^0(x) = \frac 1n \displaystyle\sum^n_{j=1} I[Z_j \leq x,\delta_j=0], \label{2.2} \\
&&H_n(x) = H_n^0(x)+H_n^1(x)=\frac 1n \displaystyle\sum^n_{j=1}
I[Z_j \leq x]. \notag
\end{eqnarray}
Asymptotic optimal nonparametric estimators of $F(x)$ and $G(x)$ are
the well-known Kaplan-Meier estimators given by
\begin{equation}
 F_n(x) = 1 - \prod_{s \leq x} \left[ 1 - \frac{H^1_n\{s\}}{{\overline H}_n(s-)} \right] \hbox{ and } \ \
 G_n(x) = 1 - \prod_{s\leq x}\left[ 1 - \frac{H^0_n \{s\} }{{\overline H}_n(s-)} \right],
\label{2.3}
\end{equation}
respectively, where an empty product is set equal to one. It can be
checked that for all $x$,
\begin{equation}
{\overline H}_n(x)={\overline F}_n(x){\overline G}_n(x). \label{2.4}
\end{equation}
Applying (2.3) and (2.4)
we get
\begin{eqnarray}
\rd F_n(x)
 =F_n(x)-F_n(x-)
 =\overline F_n(x-)\frac{\rd H_n^1(x)}{{\overline F}_n(x-){\overline G}_n(x-)}. \label{2.5}
\end{eqnarray}
It follows that
\begin{equation}
\rd H_n^1(x)={\overline G}_n(x-)\,\rd F_n(x),  \quad \rd
H_n^0(x)={\overline F}_n(x-)\,\rd G_n(x). \label {2.6}
\end{equation}

Put
\begin{eqnarray}
&& H^0(x)=P(Z\leq x, \delta=0), \quad  H^1(x)=P(Z\leq x, \delta=1), \mbox{and}  \label{2.7}\\
&&  \overline H(x)=P(Z > x). \nn
\end{eqnarray}
Then
\begin{eqnarray}
&& H^0(x)=\E\,H_n^0(x)= \int_0^x \overline F(s) \,\rd G(s),  \nn\\
&& H^1(x)=\E\,H_n^1(x) = \int_0^x \overline G(s) \,\rd F(s),\label{2.9}\\
&& \overline H(x) =\E\,\overline H_n(x)= \overline F(x) \overline
G(x). \notag
\end{eqnarray}
Here and after, the integral sign $\int^b_a$ stands for
$\int_{(a,b]}$ and $\int$ stands for $\int_0^\infty$.

\vs \non{\bf Examples of $\th$ and $g(x,\th)$} \vs

1. $g(x,\th)= I[x>y]-\th$ with $y$ fixed. Then
$$ \E\,g(Y,\th) = \int (I[x>y]-\th) \,\rd F(x)=0.$$
Solving this equation yields $\th = \overline F(y)$, the survival
function of $Y$.

2. $g(x,\th)= x^k-\th$. Then $\th = \E\,Y^k$, the $k$th moments of
$Y$.

3. $g(x,\th)= (x-t_0-\th)I[x\geq t_0]$ with $t_0$ fixed. Then
$$\th = \E (Y-t_0 |Y\geq t_0)= \frac{\E(Y-t_0)I[Y\geq t_0]}{P(Y\geq t_0)},$$
the mean residual life of $Y$.

4. $g(x,\th)= x(I[x> y] - \th) $. Then
$$\th =\frac{1}{\E\,Y}\int_y^\infty s\,\rd F(s),$$
the length biased survival function of $Y$. See Vardi (1982), for
example.

5. $g(x,\th)= x^2 - \th x$. Then
$$\th = \frac{1}{\E\,Y}\int_0^\infty x^2 \,\rd F(x),$$
the mean of the length-biased lifetime.

6. $g(x,\th)= x^2 - 2\th x$. Then
$$\th = \frac{1}{2\E\,Y}\int_0^\infty x^2 \,\rd F(x),$$
the mean of the length biased residual lifetime.

7. $g(x,\th)= I[x\leq \th]- p$ with $p\in (0,1)$. Then
$\th=F^{-1}(p)$, the $p$th quantile of $Y$.

Examples (4)-(6) often appear in renewal processes and their
applications. \vs

\section{Influence function of $\mu_n$}
\def\theequation{3. \arabic{equation}}
\setcounter{equation}{0}

Throughout Section 3, $\th$ is a fixed value. Then it is
 convenient to suppress $\th$ in the exposition, by setting   $\xi(x)=g(x,\th)$,
$\mu=\E\,g(X,\th)$,  and
\begin{eqnarray}
\mu_n = \int \xi(x) \,\rd F_n(x) = \int \frac{\xi(x)}{\overline
G_n(x-)} \,\rd H_n^1(x). \label{3.1}
\end{eqnarray}
Likewise, set $W = W(Z, \delta) =  W(Z, \delta, \th).$

In Theorem 3.1, we prove that the estimator $\mu_n$ for  $\mu$ is
asymptotic linear. That is, there is a function $W=W(Z,\del)$, such
that $\E\,W=0$, $\Var(W)<\infty$ and
$$\sqrt{n}(\mu_n-\mu)=\frac{1}{\sqrt{n}}\sum_{i=1}^n W(Z_i,\del_i) + o_p(1).$$
The function $W(Z,\del)$ is defined with respect to the true
distributions $(F,G)$. Following literature, we call $W(Z_i,\del_i)$
the $i$-th influence function of $\mu_n$. See, for e.g. van der
Vaart (1998) or Tsiatis (2006).

Theorem 3.1 will be proved by first establishing a similar result
for the truncated $W(Z_i, \delta_i)$ as defined in (3.3).
Let $\xi_b(x) =\xi(x)I[ x\leq b]$ be the restriction of $\xi(x)$ on
$(-\infty,b]$, where $b < b_H$ is an arbitrarily chosen constant. By
similar truncation,  put
\begin{equation}
\mu_b=\int \xi_b(x) \,\rd F(x), \ \ \psi_b(s)=\int_{x\geq s}
\xi_b(x) \,\rd F(x), \ \ {\overline
\delta}_i=I[Y_i>C_i]=1-\delta_i.\label{3.2}
\end{equation}
We consider the i.i.d. random variables,
\begin{equation}
W_i(b)
 =\frac{\xi_b(Z_i)\delta_i}{\overline G(Z_i)}-\mu_b
 +\frac{{\overline \delta}_i}{{\overline H}(Z_i)}\psi_b(Z_i)
 -\int\psi_b(s)\frac{I[Z_i\geq s]}{\overline H^2(s)}\,\rd H^0(s),
\label{3.3}
\end{equation}
for $i=1,\cdots,n.$

Under finite variance condition (1.3),
it can be calculated that,
\begin{eqnarray}
\E\,\frac{\xi_b(Z_i)\delta_i}{\overline G(Z_i)}
 &=&\mu_b, \ \ \E\,W_i(b)=0, \notag \\
\text{Var}(W_i(b))
 &=&\int\frac{\xi_b^2(s)}{\overline G(s)}\,\rd F(s)-\mu_b^2-\int\frac{\psi_b^2(s)}{\overline F(s)\overline G^2(s)}\,\rd G(s). \label{3.4}
\end{eqnarray}
As $b$ approaches the upper bound $ b_H$, we have
\begin{equation}
\mu_b \to \mu, \quad \psi(s) =\lim_{b\to b_H} \psi_b(s) =
\int_{x\geq s} \xi(x) \,\rd F(x), \quad \mbox{and} \label{3.5}
\end{equation}

\begin{equation}
 W_i=\lim_{b\to b_H} W_i(b)
    =\frac{\xi(Z_i)\delta_i}{\overline G(Z_i)}-\mu
     +\frac{{\overline \delta}_i}{{\overline H}(Z_i)}\psi(Z_i)
     -\int\psi(s)\frac{I[Z_i\geq s]}{\overline H^2(s)}\,\rd H^0(s), \label{3.6}
\end{equation}
for $i = 1, \cdots, n$. The $W_i's$, for $i = 1, \cdots, n,$ are
i.i.d. random variables.

Wang \& Jing (2001) use the estimating function based on
\begin{equation}
M_1(Z, \delta, \th) =\frac{\xi(Z)\delta}{\overline G(Z)} -\th
\label{1.6}
\end{equation}
to estimate $\th$ in (1.1).
Qin \& Zhao (2007) used the estimating function based on
\begin{equation}
 M_2(Z, \delta, \th) =\frac{g(Z,\th)\delta}{\overline G(Z)},   \label{1.7}
\end{equation}
where $g(x,\th)=(x-t_0-\th)I[x\geq t_0]$,  to estimate the mean
residual life $\th=\E (Y-t_0|Y\geq t_0)$ at a specified age $t_0.$
This case is covered in our formulation, see example (3) in Section
2. Comparing with $M_1$ and $M_2$, our $W_i's $ contain two
additional terms. Note that  $W_i's $ are not observable random
variables and whose estimation will be addressed in Section 4.

Under finite variance condition (1.3),
applying the dominated convergence theorem and the
Lebesgue-Stieltjes integration by parts,  we obtain
\begin{eqnarray}
\E\frac{\xi(Z_i)\delta_i}{\overline G(Z_i)}&=&\mu, \ \ \E\,W_i =0, \label{3.7} \\
\text{Var}(W_i)
 &=&\int \frac{\xi^2(s)}{\overline G(s)}\,\rd F(s) -\mu^2 -\int\frac{\psi^2(s)}{\overline F(s) \overline G^2(s)}\,\rd G(s) \label{3.7} \\
 &=&\int \frac{(\overline F(s)\xi(s)-\psi(s))^2}{\overline F^2(s)\overline G(s)} \,\rd F(s). \label{3.8}
\end{eqnarray}

{\bf Remark 3.1}\ {\it  Formulas (3.6) are (3.10)
are obtained in He $\&$ Huang (2003), and (3.11)
is given in Yang (1994). Under condition (1.3), it can be shown that
(3.10) and (3.11)
are equal.  The  variance of $W_i$ is smaller than that of $M_1$ and
$M_2$ defined in (3.7) and (3.8). The variance of the latter two
equals
$$\int \frac{\xi^2(s)}{\overline G(s)}\,\rd F(s) -\mu^2$$ with the corresponding  choices of $\xi(z)=g(z,\th)$, in $M_1$ and
$M_2$.}


 We proceed to prove Theorem 3.1. For the restricted $W_i(b)'s$, the following lemma is taken from (3.11) of He \& Huang (2003).

\vskip .10in

{\bf Lemma 3.1} \  {\it
Let $F$ and $G$ be continuous. For each  $\th$ fixed, set
$\xi(x)=g(x,\th)$. Assume that $\int \xi^2(x) \,\rd F(x)<\infty$ and
$b < b_H$. Let $\xi_b(x)$ be the restriction of $\xi $ on $(0, b]$
for $b < b_H$. Then, as $n\to \infty$,
$$ \sqrt{n}\int \xi_b(x) \,\rd (F_n(x) - F(x))
=\frac{1}{\sqrt{n}}\sum_{i=1}^n  W_i(b)+ o_p(1).$$
}

The following result will be used repeatedly and for easy reference,
it is stated in Lemma 3.2. Its proof is given the appendix.

\vskip .10in

{\bf Lemma 3.2} \ {\it
For  $b<b_H$, let $\{h_n(b)\} $ be a random sequence such that
$h_n(b) \to h(b)$ in distribution as $n\to \infty$, and
$h(b)=o_p(1)$ as $b \to b_H$. As $n\to \infty$, if $V_n=O_p(1)$ and
the random sequence $\{S_n\}$ can be written as $S_n= o_p(1)+ V_n
h_n(b)$ for any $b<b_H$, then $ S_n=o_p(1)$.
}

{\bf Remark 3.2} \ {\it
In what follows, $h_n(b)$ is  used as a generic notation to denote
any random sequence $\{h_n(b)\}$ that satisfies the assumptions of
Lemma 3.2. This simplifies many of the statements later. For
example, under condition (1.3)
and $b < b_H$, put
\begin{eqnarray}
h_n(b) = \int_b^{b_H} \frac{\xi^2(s)}{\overline G^2(s)} \,\rd
H_n^1(s). \label{3.9}
\end{eqnarray}
Then, by the SLLN and (2.8)
\begin{eqnarray}
\lim_{n\to\infty}h_n(b)
 &=&\lim_{n\to\infty}\int_b^{b_H}\frac{\xi^2(s)}{\overline G^2(s)} \,\rd H_n^1(s)
   =\int_b^{b_H} \frac{\xi^2(s)}{\overline G^2(s)} \,\rd H^1(s) \nonumber \\
 &=&\int_b^{b_H} \frac{\xi^2(s)}{\overline G^2(s)} \overline G(s) \,\rd F(s)
   = h(b) \to 0, \ \mbox{as $b\to b_H.$} \nonumber
\end{eqnarray}
}

\begin{thm}
Let $W_i$ be given by (3.6).
 Suppose
$F$ and $G$ are continuous, and for each fixed $\th$
 set  $\xi(x)=g(x,\th)$. Then under condition (1.3)
 as $n\to \infty$,
\begin{equation}
\sqrt{n}\int \xi(x) \,\rd (F_n(x)
-F(x))=\frac{1}{\sqrt{n}}\sum_{i=1}^n W_i+ o_p(1).\notag
\end{equation}
\end{thm}

\begin{proof} For $b<b_H$, put $ \overline \xi_b=\xi(x)
-\xi_b(x)=\xi(x)I[x> b]$. Decompose the following difference as
\begin{equation}
\sqrt{n} \int \xi(x) \,\rd
(F_n(x)-F(x))-\frac{1}{\sqrt{n}}\sum_{i=1}^n W_i \equiv
J_1(b)+J_2(b) -J_3(b), \label{3.10}
\end{equation}
where
\begin{eqnarray}
&& J_1(b)=\sqrt{n} \int \xi_b(x)  \,\rd  (F_n(x)-F(x))-
\frac{1}{\sqrt{n}}\sum_{i=1}^n W_i(b), \notag \\
&& J_2(b)= \sqrt{n}\int \overline \xi_b(x)  \,\rd (F_n(x)-F(x)),\notag \\
&&J_3(b)=\frac{1}{\sqrt{n}}\sum_{i=1}^n (W_i-W_i(b)). \notag
\end{eqnarray}
Let $\overline \psi_b(x)=\int_{s\geq x} \overline \xi_b(s) \,\rd
F(s)$. It follows that
\begin{equation}
 \overline \psi_b^2(x) \leq \overline F(x)^2\int_{s\geq x}
\overline \xi_b^2(s) \,\rd F(s).\label{3.11}
 \end{equation}

By Lemma 3.1, $J_1(b)=o_p(1)$. We shall show that $J_2(b)=h_n(b)$
and $J_3(b)=h_n(b)$ as $n \to \infty$. Applying Corollary 1 of Yang
(1994), $J_2(b)$ converges weakly to $N(0,\,\overline \sig_b^2)$,
where $\overline \sig_b^2$ is similarly defined as in  (3.12) with
$\xi$ and $\psi$ replacing by their restrictions $\overline \xi_b$
and $\overline \psi_b$ respectively.

Now
\begin{eqnarray}
\overline \sig_b^2 \leq  \int \frac{\overline \xi_b^2(x)}{\overline
G(x)}\,\rd F(x) \to 0, \ \text{as} \  b \to b_H. \label{3.12}
\end{eqnarray}
Therefore  $J_2(b)$ converges to $h(b)=Z_0\overline \sig_b$ in
distribution, and $h(b)=o_p(1)$ as $b\to b_H$, where $Z_0$ is a
$N(0,\,1)$ random variable.  It follows that  $J_2(b)=h_n(b)$.

To prove $J_3(b)=h_n(b)$, note that the difference $ W_i-W_i(b) $,
as given in (3.6) and (3.3)
equals to
\begin{equation}
\frac{\overline\xi_b(Z_i)\delta_i}{\overline G(Z_i)}
 -\int\overline\xi_b(x)\,\rd F(x)
 +\frac{{\overline\delta}_i}{{\overline H}(Z_i)}\overline \psi_b(Z_i)
 -\int\overline\psi_b(s)\frac{I[Z_i\geq s]}{\overline H^2(s)}dH^0(s).
\label{3.13}
\end{equation}
Therefore $ W_i-W_i(b) $ are i.i.d. random variables with mean zero
and variance $\overline \sig_b^2$. Hence,  $J_3(b)=h_n(b)$ follows
for the same reason as that of $J_2(b)$.

We conclude that the following holds for (3.13),
$$ \sqrt{n}\int\xi(x) \,\rd (F_n(x)-F(x))-\frac{1}{\sqrt{n}}\sum_{i=1}^n W_i=o_p(1)+h_n(b)+h_n(b).$$
Theorem 3.1 follows from Lemma 3.2.
\end{proof}


{\bf Remark 3.3}\ {\it
If $b_F < b_G$, then (1.3)
is equivalent to $\xi$ having finite second moment. If $\xi$ is
bounded and away from zero, then (1.3)
is equivalent to $\int_0^{\infty} \,\rd F(s)/{\overline G(s)}.$
}

\section{Empirical Likelihood Ratios and Confidence Intervals for $\th$}
\def\theequation{4. \arabic{equation}}
\setcounter{equation}{0}

To develop an EL inference procedure, we consider a specific
$g(x,\th)$. For each fixed $\th$, as before, set $\xi(x)=g(x, \th)$.
We shall utilize the i.i.d. random variables
\begin{eqnarray}
 W_i=\frac{\xi(Z_i)\delta_i}{\overline G(Z_i)}
     +\frac{{\overline \delta}_i}{{\overline H}(Z_i)}\psi(Z_i)
     -\int\psi(s)\frac{I[Z_i\geq s]}{\overline H^2(s)}dH^0(s). \label{4.1}
\end{eqnarray}
to obtain an estimating equation for the EL ratio. Recall that
$\mu=\E\,W_i=\int_0^\infty\xi(x)\,\rd F(x)$  and $\text{Var}(W_i)$
are given by (3.10).
Note that setting $\xi(x)=g(x,\th)$ above has nothing to do with
defining $\th$ from the equation $\E\,g(X,\th)=0$ as given in (1.5).
If, however, the true parameter $\th_0$ is the  solution of the
equation
\begin{equation}
 \E\,g(Y,\th)=\int g(x,\th)\,\rd F(x)=0, \label{4.2}
\end{equation}
then $\xi(x)=g(x,\th_0)$ is such that
$$\mu = \E\,W_i=\int_0^\infty\xi(x)\,\rd F(x)=0.$$

Regarding $W_i$ for $i=1,\cdots,n$ as a ``complete'' random sample,
one could formulate an EL likelihood ratio $R(\th_0)$ with
multinomial probability $p_i$ assigned to $W_i$ and the constraint
$\sum_{i=1}^n W_i p_i = 0$.
 However, $W_i's$ are not
observable because of the unknown distributions $G$, $F$ $H$ and
$H^0$. We shall replace them by the KM estimates, $F_n$, $G_n$ given
by (2.3)
and an estimate of $\psi$,
\begin{equation}
\psi_n(x)=\int_{s\geq x}\xi(s)\,\rd F_n(s). \label{4.3}
\end{equation}
Replacing $\overline G, \overline H, H^0$ in (4.1) by their
corresponding empirical distributions in (2.2)  gives an
approximation of $W_i$ in (4.1)
by
\begin{equation}
W_{ni}
 =\frac{\xi(Z_i)\delta_i}{\overline G_n(Z_i-)}
 +\frac{{\overline \delta}_i}{{\overline H_n}(Z_i-)}\psi_n(Z_i)
 -\int\psi_n(s)\frac{ I[Z_i\geq s]}{\overline H_n^2(s-)}\,\rd H_n^0(s).
\label{4.4}
\end{equation}
The price to pay for the estimation is that ${W_{ni}}'s $ are not
stochastically independent which complicates the ensuing analysis.


The following theorem indicates the possibility of using $W_{ni}$ to
construct  empirical likelihood ratio and to obtain  asymptotically
a standard $\chi^2$ distribution.

\begin{thm}
Let $W_{ni}$ be given by (4.4)
and $\E\, \xi(Y) =0$. Then under condition (1.3),
as $n\to \infty$, we have
\begin{equation}
\frac{1}{\sqrt{n}}\sum_{i=1}^n W_{ni} =
\frac{1}{\sqrt{n}}\sum_{i=1}^n W_{i} + o_p(1). \end{equation}
\end{thm}

\begin{proof} By (2.6),
we have
\begin{eqnarray*}
& & \frac{1}{n}\sum_{i=1}^{n}W_{ni} \\
 &=&\frac{1}{n}\sum_{i=1}^{n} \Big(\frac{\xi(Z_i)\delta_i}{\overline G_n(Z_i-)}
    + \frac{{\overline \delta}_i}{{\overline H_n}(Z_i-)}\psi_n(Z_i)
    -\int \psi_n(s)\frac{I[Z_i\geq s]}{\overline H_n^2(s-)}dH_n^0(s)\Big)\notag\\
 &=&\int\frac{\xi(s)}{\overline G_n(s-)} \,\rd H^1_n(s)
    + \int\frac{\psi_n(s) }{{\overline H_n}(s-)} \,\rd H_n^0(s)
    - \int\psi_n(s)\frac{\overline H_n(s-)}{\overline H_n^2(s-)}dH_n^0(s) \notag\\
&=& \int \frac{\xi(s)} {\overline G_n(s-)} \,\rd H^1_n(s) =  \int
\xi(s) \,\rd F_n(s).
\end{eqnarray*}
Applying   $\int_0^\infty \xi(s) \, \rd F(s)=\E \xi(Y)  =0$  and
Theorem 3.1, we arrive at
\begin{equation}
\frac{1}{\sqrt{n}}\sum_{i=1}^{n}W_{ni}= \sqrt{n} \int \xi(s) \,\rd
(F_n(s) -F(s)) =\frac{1}{\sqrt{n}}\sum_{i=1}^n W_{i} + o_p(1)
\label{h4.6}.
\end{equation}

\end{proof}

Following Owen (2001), define the EL ratio of $\th$ by a multinomial
likelihood subject to constraints as
\begin{equation}
R(\th)=\sup_{\{p_i\}}\left\{\ \prod_{i=1}^{n}np_{i}\Big|\
\sum_{i=1}^{n}p_{i}=1,\,\sum_{i=1}^{n}p_{i}W_{ni}=0,\,p_{i}\geq
0,\,i=1,2,\cdots,n\right\}. \label{4.5}
\end{equation}
To determine $R(\th)$, we solve, as usual, for the Lagrange
multipliers $\mu$ and $\lambda$  in
$$A=\sum_{i=1}^n \log(np_i) - n\lambda(\sum_{i=1}^n p_iW_{ni})-\mu(1-\sum_{i=1}^n p_i).$$
Then $\mu=-n$ and $p_{i}=\dfrac{1}{n}(1+\lambda W_{ni})^{-1},\
i=1,2,\cdots,n,$ where  $\lambda$ is the solution of
$$\dfrac{1}{n}\sum\limits_{i=1}^{n}\dfrac{W_{ni}}{1+\lambda W_{ni}}=0.$$

The uniqueness of $\lambda$ will be addressed in the proof of
Theorem 4.2. The EL ratio of $\th$ can be written as
\begin{equation}
R(\th)=\prod_{i=1}^{n}(np_{i})=\prod_{i=1}^{n}(1+\lambda
W_{ni})^{-1}. \label{4.6}
\end{equation}

\begin{thm}  Suppose that $\th_0$ is the
unique solution of (4.2)
and finite second moment (1.3) holds. Set $\xi(x)=g(x,\th_0)$. Then
$l(\th_0)=-2\log R(\th_0)$ converges in distribution  to a
$\chi_1^2$ random variable with one degree of freedom,  as $n\to
\infty$.
\end{thm}
Applying Theorem 4.2,  confidence intervals for $\th$ can be
constructed as
\begin{equation}
I_{1}=\{\th: l(\th)\leq c_{1-\alpha}\}, \label{4.7}
\end{equation}
where $c_{1-\alpha}$ is the $(1-\alpha)$th quantile of the
$\chi_1^{2}$ distribution.  $I_{1}$ has asymptotic coverage
probability of $1-\alpha$, as $n \to \infty.$

\vskip .12in

To prove Theorem 4.2, we shall make use of the following Taylor's
expansion of $-2\log R(\th_0)$,
\begin{eqnarray*}
-2\log R(\th_0)
 &=& 2\sum_{i=1}^n \ln(1+\la_n W_{ni})\\
 &=&\sum_{i=1}^n 2\Big(\la_n W_{ni}-\frac{1}{2}\la_n^2 W_{ni}^2+\eta_i\Big) \qquad  (\ \text{here} \ |\eta_i|\leq |X_i|^3 ) \\
 &=&2\la_n n \overline W_n -\la_n^2 n S_n^2 + 2 \sum_{i=1}^n \eta_i\\
 &=&2n \frac{\overline W_n^2}{\sig^2}-\frac{1}{\sig^4}n\overline W_n^2(\sig^2+o_p(1)) + o_p(1)\\
 &=&\frac{n \overline W_n^2}{\sig^2}+o_p(1) \to \chi_1^2, \qquad \text{in distribution},
\end{eqnarray*}
where
\begin{eqnarray}
\overline W_n = \sum_{i=1}^n W_{ni}, \quad \mbox{and \; $S_n^2 =
\sum_{i=1}^n W_{ni}^2$.}
\end{eqnarray}
Asymptotic analysis of $\overline W_n = \sum_{i=1}^n W_{ni}$ and
$S_n^2 = \sum_{i=1}^n W_{ni}^2$ are needed for establishing the last
two equalities in the expansion. It will be proven in Lemma 4.3 and
Theorem 4.2 that both of these averages are related to the
asymptotic variance (1.4) or (3.10). The following lemmas are needed
for proving Theorem 4.2.

{\bf Lemma 4.1} \ {\it
Let $f_n(x)$  and $f(x)$ be monotone functions defined on the range
of $Z$. If $f(x)$ is continuous and for $x\in [0, b_H]$, $f_n(x) \to
f(x)$, as $n\to \infty$, then $f_n(x)$ converges to $f(x)$ uniformly
on $[0, b_H]$.}
The proof is omitted.

{\bf Lemma 4.2} \ {\it
Let $V_{ni}=W_{ni}{\overline H}_n(Z_i-) {\overline H}(Z_i)$ and
$V_{i}=W_{i}{\overline H}^2(Z_i)$. Under the conditions of Theorem
4.2 and $\th=\th_0$, as $n\to \infty$,

(1) \ $\dfrac{1}{n}\sum\limits_{i=1}^{n}(W_{ni} -W_i)^2 = o_p(1)$,

(2) \ $\dfrac{1}{n}\sum\limits_{i=1}^{n}(V_{ni} -V_i)^2 \to 0, \ \text{a.s.}$.\\
}

The proof is relegated to the Appendix. Lemma 4.2 is needed for
showing that with probability 1, for large $n$ the set $\{W_{ni}\}$
contains a positive and a negative value. To facilitate the proof,
$V_{ni}'s$, a modification of $W_{ni}'s$, are introduced to deal
with  the problem at the boundary $b_H.$ It follows that for large
$n$ there exists a unique $\lambda_n$ for $R(\th)$ in (4.8).

{\bf Lemma 4.3}\ {\it
Let $W_{ni}$ and $W_i$ be given by (4.4) and (4.1),
respectively. Under the conditions of Theorem 4.2 and $\th=\th_0$,
as $n\to \infty$,
$$\max\limits_{1\leq i\leq n}|W_{ni}|=o_{p}(\sqrt{n}), \ \ \ \frac{1}{n}\sum_{i=1}^{n}W^2_{ni}=\sig^2 +o_p(1),$$
and
$$\frac{1}{\sqrt{n}}\sum_{i=1}^{n}W_{ni} \to N(0,\,\sigma^2),\ \ \text{in dist.},$$
where $\sig^2=\Var(W_i)$ is given by (3.10).}

The proof is relegated to the Appendix.  We now prove Theorem 4.2.

\renewcommand{\proofname}{Proof of Theorem 4.2}

\begin{proof} The Lagrange multiplier $\la$ in (4.6) appears in the equation
\begin{eqnarray}
h(\la)= \frac{1}{n} \sum_{i=1}^n \frac{W_{ni}}{1+\la W_{ni}}=0.
\label{4.8}
\end{eqnarray}
We shall show that for large $n$, $h(\la) = 0$ has a unique solution
$\la_n$ such that  $\la_n W_{ni} > -1$ for all $i$. Put
$$U_{i}=\begin{cases} -W_{ni}^{-1}, & W_{ni}\neq 0, \\
 \infty, & W_{ni}=0. \end{cases}$$ Let $U_{(1)}\leq U_{(2)}\leq \cdots \leq
U_{(n)}$ be the ordered statistics of $U_1,U_2,\cdots, U_n$. Then
$$h(\la)=\frac{1}{n}\sum_{i=1}^n \frac{W_{ni}}{1+\la W_{ni}}= \frac{1}{n}\sum_{i=1}^n\frac{1}{\la-U_{(i)}}$$
is monotone  and differentiable in $\la$ on each nonempty interval
$(U_{(i)}, U_{(i+1)})$. We claim that for large $n$, there exists an
$i $ such that $U_{(i)}<0< U_{(i+1)}$. To see this, we note that for
every $\vv>0$, by Lemma 4.2(2)
$$\frac{1}{n}\sum_{i=1}^n I[\, |V_{ni}-V_i|\geq \vv \, ] \leq \frac{1}{n}\sum_{i=1}^n (V_{ni}-V_i)^2/ \vv^2=o(1),\ \text{a.s.}.$$
Using the fact that $I[V_i\geq \vv ]\leq I[V_{ni}\geq
\vv/2]+I[{|V_{ni}-V_i| \geq \vv/2}]$, we get
$$\frac{1}{n}\sum_{i=1}^n I[ V_{ni}\geq \vv/2]\geq\frac{1}{n}\sum_{i=1}^n I[ V_i\geq \vv]+o(1), \ \text{a.s.}.$$
Using the fact that $P(V_1>0)=P(W_1>0)>0$, it is seen that for some
$\vv>0$,
\begin{eqnarray*}
\liminf_{n\to\infty} \frac{1}{n}\sum_{i=1}^n I[ V_{ni}\geq \vv/2]
 \geq P(V_1\geq \vv)>0, \ \text{a.s.}.
\end{eqnarray*}
Similarly, we have
$$\liminf_{n\to\infty} \frac{1}{n}\sum_{i=1}^n I[ V_{ni}\leq - \vv/2 ]
\geq P(V_1\leq -\vv)>0, \ \text{a.s.}.$$ Since $W_{ni}$ and $V_{ni}$
have the same sign, hence the claim is true. It follows that there
is a unique $\la_n \in (U_{(i)}, U_{(i+1)})=(-1/\max\{W_{ni}\},
-1/\min\{W_{ni}\})$ such that $h(\la_n)=0$ and $\la_n
\max\{W_{ni}\}>-1$ and $\la_n \min\{W_{ni}\}>-1 $.

The rest of the proof is similar to that of Owen (2001). In fact,
setting
$$X_i=\la_n W_{ni}, \ \ \overline W_n= \frac{1}{n}\sum_{i=1}^n W_{ni}, \ \ S_n^2 =  \frac{1}{n} \sum_{i=1}^n W_{ni}^2,$$
we have $S_n^2=\sig^2+o_p(1)$,
$$\frac{1}{n}\sum_{i=1}^n\frac{X_i^2}{1+X_i}=\frac{1}{n}\sum_{i=1}^n\frac{(X_i+1-1)X_i}{1+X_i}=\overline X_n-\la_nh(\la_n)=\la_n \overline W_n$$
and
\begin{equation}
\la_n^2S_n^2
 =\frac{1}{n}\sum_{i=1}^n X_i^2 \leq \frac{1}{n}\sum_{i=1}^n\frac{X_i^2}{1+X_i}\big(1+\max_{1\leq j \leq n} |X_j|\big )
 =\la_n \overline W_n + \la_n^2 o_p(1).\label{4.9}
\end{equation}
It follows that
$$\la_n =\frac{ \overline W_n}{\sig^2+o_p(1)}=O_p(n^{-1/2})$$
and
\begin{equation} \overline W_n= \la_n \sig^2 +o_p(n^{-1/2}). \label{4.10}
\end{equation}
Applying Lemma 4.3, we have
$$\sum_{i=1}^n |X_i|^3 \leq \la_n^3 \sum_{i=1}^n |W_{ni}|^2 \max_{1\leq j\leq n}|W_{ni}|=O_p(n^{-3/2})O_p(n)o_p(n^{1/2})=o_p(1).$$

Therefore the Taylor expansion (above (4.10)) is valid from which
the theorem follows.

\end{proof}

{\bf Remark 4.1} \ {\it
We are able to obtain the standard asymptotic $\chi^2_1$
distribution  for $-2\log R(\th_0)$ is because
 the asymptotic variance of
$$\frac{1}{\sqrt{n}}\sum_{i=1}^{n}W_{ni}=\frac{1}{\sqrt{n}}\sum_{i=1}^{n}W_{i}+
o_p(1)
$$
(which is $\sigma^2=\Var(W_i)$) equals the limit of (see Lemma 4.3)
$$\frac{1}{n}\sum_{i=1}^{n}W^2_{ni}= \frac{1}{n}\sum_{i=1}^{n}W^2_{i}+ o_p(1).$$

If  $-2\log(EL \, ratio)$ is based on the estimating function
$M_2=M_2(Z,\delta,\th)$ in (3.8) ( or in (3.7)),
then
$$
{V}_{ni}=\frac{g(Z_i, \th)\delta_i}{1-G_n(Z_i)}
$$
will be used to construct $-2\log(EL\, ratio)$. Now, the asymptotic
variance of
$$\frac{1}{\sqrt{n}}\sum_{i=1}^{n}V_{ni}=\frac{1}{\sqrt{n}}\sum_{i=1}^{n}W_{i}+
o_p(1)
$$
is $\sigma^2$ (see(3.10)),
but the limit of $ \frac{1}{n}\sum_{i=1}^{n}V^2_{ni}$ or
$$ \frac{1}{n}\sum_{i=1}^{n}\Big( V_{ni} - \frac{1}{n}\sum_{j=1}^{n} V_{nj} \Big)^2$$ is
(see Remark 3.1)
$$\sigma_1^2 = \int \frac{\xi^2(s)}{\overline G(s)}\,\rd F(s) -\mu^2, \ \ \mu =0.$$
Therefore, a scaled parameter $r=\sig_1^2/\sig^2$ must be introduced
in order to obtain the asymptotic distribution for $-2\log(EL\,
ratio)$  as in Wang $\&$ Jing (2001).
}

\section{Simulation}
\def\theequation{5. \arabic{equation}}
\setcounter{equation}{0}

 Simulations are carried out to study and compare finite sample performance of confidence intervals $I_1$
 in (4.9)
derived from Theorem 4.2 and $I_2$ from the scaled $\chi_1^2$
distribution given by Wang \& Jing (2001) and Qin \& Zhao (2007).

To calculate $I_1$, $W_{ni}$ in (4.4)
is used, where
$$
W_{ni}
 =\frac{\xi(Z_i)\delta_i}{\overline G_n(Z_i-)}
 +\frac{{\overline \delta}_i}{{\overline H_n}(Z_i-)}\psi_n(Z_i)
 - \frac{1}{n}\sum_{j=1}^n\psi_n(Z_j)\frac{ I[Z_i\geq Z_j]{\overline \delta}_j}{\overline
 H_n^2(Z_j-)},
$$
and $\psi_n(x)$ is given by (4.3).

 Confidence
intervals $I_2$  are calculated as follows. Let $F_n,$ and $G_n$ be
the Kaplan-Meier estimators defined by (2.3).
 Suppose $\hat{\theta}$ is the unique solution of
$\int g(s,\,\theta)\,\mathrm{d} F_n(s)=0$. Set
\begin{equation*}
\xi_i=g(Z_i,\,\theta),\qquad\hat{\xi}_i=g(Z_i,\,\hat{\theta}),
\end{equation*}
\begin{equation}
 V_{ni}=\frac{\xi_i\,\delta_i}{1-G_n(Z_i)},\qquad\hat{V}_{ni}=\frac{\hat{\xi}_i\,\delta_i}{1-G_n(Z_i)},
 \label{5.1}
\end{equation}
\begin{equation*}
 \sigma_1^2=\frac{1}{n}\sum_{i=1}^n (\hat{V}_{ni} -\overline{V}_n)^2,\qquad\overline{V}_n = \frac{1}{n}\sum_{i=1}^n\hat{V}_{ni},
\end{equation*}
\begin{equation*}
 \hat{r}=\frac{\sigma_1^2}{n\,\widehat{\mathrm{Var}}^*(jack)},
\end{equation*}
where  $n\,\widehat{\mathrm{Var}}^*(jack)$ is the modified jackknife
estimator of the asymptotic variance of $\hat{\xi}$ given in
Stute(1996). Then, the EL-based confidence interval for $\th$ is
\begin{equation}
I_2=\left\{\theta:\,\,2\,\hat{r}\sum_{i=1}^n \log(1+\lambda V_{ni})
\leq c_{1-\alpha} \right\}, \label{5.2}
\end{equation}
where $\lambda$ is the solution of $ \sum_{i=1}^n V_{ni}/(1+\lambda
V_{ni})=0$.

\vskip 0.10in

 Simulations were performed in two scenarios. In scenario I, the parameter of interest is  $\th_0=\E\,Y$ and in scenario II, the mean residual lifetime.
\vs {\bf Scenario I:}
 The parameter of interest is $\th_0=\E\,Y$  and  $\xi(x)=g(x,\th)=x-\th$ is used for calculating  $I_1$. Two cases were simulated:

(i) The lifetime $Y$ is uniformly distributed on $(0,1)$ and the
censoring time $C$ is uniformly distributed on (0,\, c).  We
selected $c=2.5$ and $c=1.3$ which corresponds respectively to 20\%
and 30\% censoring proportions.

(ii)  $Y$ has a Weibull(1,\,10) distribution and $C$ has an
Exp($\lambda$) distribution. Then for $\lambda=4.3$ and
$\lambda=2.7$, the corresponding censoring proportions are 20\% and
30\%. The simulated observations are  $n$  i.i.d. copies of
$Z=\min(Y,\,C), \delta=I[Y \leq C]$. Based on the simulated
observations, confidence intervals $I_1$ derived from Theorem 4.2
and $I_2$ from (5.2)
were calculated. The process was repeated for $N=2\times 10^4$ times
and the coverage proportions and the average width of the confidence
intervals were calculated using the $N$ data sets. The results are
summarized in Table 1 and Table 2.

\begin{table*}
\caption{The coverage proportions for  the true $\th_0=\E\,Y$}
\begin{tabular}{ccccccc}
\hline \multicolumn{2}{l}{20\% censoring proportion}\\ \cline{1-7}
 Nominal Value &  Sample Size   & \multicolumn{2}{c}{Uniform(0,\,1)}& & \multicolumn{2}{c}{Weibull(1,\,10)}\\
  $1-\alpha$   &     $n$        & $I_2$ & $I_1$ &  & $I_2$ & $I_1$  \\\hline
   0.90        &  20            & 0.876 & 0.881 &  & 0.871 & 0.871  \\
               &  40            & 0.895 & 0.897 &  & 0.889 & 0.890  \\
               &  60            & 0.897 & 0.897 &  & 0.893 & 0.893  \\
               &  80            & 0.897 & 0.898 &  & 0.896 & 0.896  \\\cline{1-7}
   0.95        &  20            & 0.928 & 0.935 &  & 0.922 & 0.924  \\
               &  40            & 0.946 & 0.949 &  & 0.939 & 0.941  \\
               &  60            & 0.947 & 0.948 &  & 0.945 & 0.946  \\
               &  80            & 0.947 & 0.947 &  & 0.947 & 0.948  \\\cline{1-7}
\multicolumn{2}{l}{30\% censoring proportion}\\ \cline{1-7}
 Nominal Value &  Sample Size   & \multicolumn{2}{c}{Uniform(0,\,1)}& & \multicolumn{2}{c}{Weibull(1,\,10)}\\
  $1-\alpha$   &     $n$        & $I_2$ & $I_1$ &  & $I_2$ & $I_1$  \\\hline
   0.90        &  20            & 0.841 & 0.861 &  & 0.867 & 0.869  \\
               &  40            & 0.885 & 0.890 &  & 0.890 & 0.891  \\
               &  60            & 0.888 & 0.892 &  & 0.890 & 0.891  \\
               &  80            & 0.897 & 0.900 &  & 0.893 & 0.894  \\\cline{1-7}
   0.95        &  20            & 0.897 & 0.916 &  & 0.916 & 0.924  \\
               &  40            & 0.934 & 0.941 &  & 0.939 & 0.943  \\
               &  60            & 0.941 & 0.946 &  & 0.944 & 0.946  \\
               &  80            & 0.945 & 0.947 &  & 0.945 & 0.947  \\\cline{1-7}
\end{tabular}
\end{table*}

\begin{table*}
\caption{The average width of confidence intervals for
$\th_0=\E\,Y$}
\begin{tabular}{ccccccc}
\hline \multicolumn{2}{l}{20\% censoring
proportion}&\multicolumn{2}{c}{width} &
&\multicolumn{2}{c}{width}\\\cline{1-7}
 Nominal Value & Sample Size     &\multicolumn{2}{c}{Uniform(0,\,1)} &\,& \multicolumn{2}{c}{Weibull(1,\,10)}\\
 $1-\alpha$    &      $n$        & $I_2$  & $I_1$ &    & $I_2$ & $I_1$   \\\hline
 0.90          &  20             & 0.217  & 0.218 &    & 0.092 & 0.091   \\
               &  40             & 0.157  & 0.157 &    & 0.066 & 0.065   \\
               &  60             & 0.129  & 0.129 &    & 0.054 & 0.053   \\
               &  80             & 0.112  & 0.112 &    & 0.046 & 0.046   \\\cline{1-7}
 0.95          &  20             & 0.258  & 0.259 &    & 0.110 & 0.109   \\
               &  40             & 0.187  & 0.187 &    & 0.079 & 0.078   \\
               &  60             & 0.154  & 0.154 &    & 0.064 & 0.064   \\
               &  80             & 0.133  & 0.133 &    & 0.056 & 0.055   \\\cline{1-7}
\multicolumn{2}{l}{30\% censoring proportion}&
\multicolumn{2}{c}{width} & & \multicolumn{2}{c}{width}\\\cline{1-7}
 Nominal Value & Sample Size     &\multicolumn{2}{c}{Uniform(0,\,1)} &\,& \multicolumn{2}{c}{Weibull(1,\,10)}\\
 $1-\alpha$    &      $n$        & $I_2$  & $I_1$ &    & $I_2$ & $I_1$   \\\hline
 0.90          &  20             & 0.220  & 0.227 &    & 0.097 & 0.096   \\
               &  40             & 0.162  & 0.164 &    & 0.069 & 0.069   \\
               &  60             & 0.134  & 0.134 &    & 0.057 & 0.057   \\
               &  80             & 0.116  & 0.116 &    & 0.049 & 0.049   \\\cline{1-7}
 0.95          &  20             & 0.260  & 0.270 &    & 0.116 & 0.116   \\
               &  40             & 0.192  & 0.196 &    & 0.083 & 0.083   \\
               &  60             & 0.159  & 0.160 &    & 0.068 & 0.068   \\
               &  80             & 0.138  & 0.139 &    & 0.059 & 0.059   \\\cline{1-7}
\end{tabular}
\end{table*}

\vskip 0.10in

The following are noted.

(1) As the sample size $n$ increases, all of the coverage
proportions converge to the nominal level $1- \alpha$.

(2) For Uniform(0,\,1) distribution,  $I_1$ has  better coverage
proportions. In 8/16 of the cases, the average width of $I_2$ is
slightly shorter than that of $I_1$. In 8/16 of the cases, $I_2$ and
$I_1$ have the same average width.

(3) For Weibull(1,\,10) distribution, $I_1$ has better coverage
proportion and width.

\vskip .1in

In the $j$th simulation,  $\{W_{ni}\}$ and $\{V_{ni}\}$ were
calculated according to (4.4)
and (5.1)
respectively. Then the  sample means of $\{W_{ni}\}$ and
$\{V_{ni}\}$ are the same (see the proof of Theorem 4.1).
But the sample variance $s_W^2(j)$ of $\{W_{ni}\}$ and the sample
variance $s_V^2(j)$ of $\{V_{ni}\}$ are different. Let
$$s_W^2 =\frac{1}{N} \sum_{j=1}^Ns_W^2(j), \ \ \ s_V^2 =\frac{1}{N}\sum_{j=1}^Ns_V^2(j).$$
Table 3 shows that the sample variance  of $\{W_{ni}\}$ is smaller
than that of  $\{V_{ni}\}$. This is proved in Remark 3.1 for the
population variances.

\begin{table*}
\caption{The sample variances of $\{W_{ni}\}$ and $\{V_{ni}\}$,
$\th_0=\E\,Y$}
\begin{tabular}{cccccc}
\hline \multicolumn{6}{l}{20\% censoring proportion}\\\cline{1-6}
                  &\multicolumn{2}{c}{Uniform(0,\,1)} &\,& \multicolumn{2}{c}{Weibull(1,\,10)}\\
  Sample Size $n$ &   $s_W^2$   &  $s_V^2$  &  & $s_W^2$  &   $s_V^2$   \\\hline
   20             &    0.0935    &    0.1121  &  &   0.0157  &   0.0163  \\
   40             &    0.0938    &    0.1115  &  &   0.0157  &   0.0162  \\
   60             &    0.0937    &    0.1107  &  &   0.0157  &   0.0161  \\
   80             &    0.0934    &    0.1100  &  &   0.0158  &   0.0162  \\\cline{1-6}
   \multicolumn{6}{l}{30\% censoring proportion}\\\cline{1-6}
                  &\multicolumn{2}{c}{Uniform(0,1)} &\,& \multicolumn{2}{c}{Weibull(1,\,10)}\\
  Sample Size $n$ &  $s_W^2$    &  $s_V^2$  &  &  $s_W^2$ &    $s_V^2$   \\\hline
   20             &    0.1005    &    0.1386  &  &   0.0175  &   0.0185  \\
   40             &    0.1013    &    0.1401  &  &   0.0176  &   0.0184  \\
   60             &    0.1016    &    0.1402  &  &   0.0176  &   0.0183  \\
   80             &    0.1012    &    0.1393  &  &   0.0176  &   0.0183  \\
\cline{1-6}
\end{tabular}
\end{table*}

\vskip .1in

{\bf Scenario II:}  Let $g(x,\th)=(x-t_0 -\th)I[x\geq t_0]$. Then by
solving the equation
  $\E\,g(Y,\th)=0$, we obtain the mean residual life of $Y$,
\begin{equation}
 \th_0= \E (Y-t_0 |Y\geq t_0)= \frac{\E(Y-t_0)I[Y\geq t_0]}{P(Y\geq t_0)}, \label{5.3}
\end{equation}
as studied in Qin \& Zhao (2007). Let  $Y$ have a Weibull\,(1,\,10)
distribution and $C$ have an Exp($\lambda$) distribution. By setting
 $\lambda=4.3$ and $\lambda=2.7$, we achieved 20\% and 30\% censoring proportions respectively. As in Scenario I, each simulation was repeated  $N=2\times 10^4$
times. The coverage proportion of the $N$ data sets and their
average width were calculated. The results are summarized in Table 4
and Table 5, respectively.
\begin{table*}
\caption{The coverage proportion and average width of confidence
intervals for $\th_0 =  \E (Y-t_0 |Y\geq t_0)$ under the assumptions
of $Y \sim $ Weibull(1,\,10), 20\% censoring proportion, and
$1-\alpha=0.90$. }
\begin{tabular}{ccccccccccc}
\hline
 Sample Size $n$  &  Method       & \multicolumn{4}{c}{Coverage Ratio} &\,& \multicolumn{4}{c}{Average Width}\\\cline{1-11}
                  &               &     \multicolumn{4}{c}{$P(Y \geq t_0)$}     &\,& \multicolumn{4}{c}{$P(Y \geq t_0)$}     \\
                  &               & $0.90$ & $0.70$ & $0.50$ & $0.30$ &  & $0.90$ & $0.70$ & $0.50$ & $0.30$ \\\hline
 20               &  $I_2$        & 0.878  & 0.851  & 0.795  & 0.659  &  & 0.074  & 0.062  & 0.054  & 0.044  \\
                  &  $I_1$        & 0.881  & 0.863  & 0.820  & 0.701  &  & 0.074  & 0.062  & 0.056  & 0.048  \\
\cline{1-11}
 40               &  $I_2$        & 0.889  & 0.878  & 0.859  & 0.800  &  & 0.053  & 0.046  & 0.042  & 0.039  \\
                  &  $I_1$        & 0.891  & 0.884  & 0.874  & 0.833  &  & 0.053  & 0.046  & 0.043  & 0.041  \\
\cline{1-11}
 60               &  $I_2$        & 0.897  & 0.892  & 0.877  & 0.839  &  & 0.044  & 0.037  & 0.035  & 0.034  \\
                  &  $I_1$        & 0.898  & 0.897  & 0.888  & 0.863  &  & 0.044  & 0.038  & 0.035  & 0.035  \\
\cline{1-11}
 80               &  $I_2$        & 0.895  & 0.888  & 0.884  & 0.853  &  & 0.038  & 0.033  & 0.031  & 0.030  \\
                  &  $I_1$        & 0.896  & 0.892  & 0.892  & 0.871  &  & 0.038  & 0.033  & 0.031  & 0.031  \\
\cline{1-11}
\end{tabular}
\end{table*}
\begin{table*}
\caption{The coverage ratio and average width of confidence
intervals for  $\th_0= \E (Y-t_0 |Y\geq t_0)$, under the assumptions
$Y\sim$ Weibull (1,\,10), 30\% censoring proportion and
$1-\alpha=0.90$.}
\begin{tabular}{ccccccccccc}
\hline
 Sample Size $n$  &  Method & \multicolumn{4}{c}{Coverage Ratio} &\,& \multicolumn{4}{c}{Average Width }\\\cline{1-11}
                 &               &     \multicolumn{4}{c}{$P(Y \geq t_0)$}      &\,& \multicolumn{4}{c}{$P(Y \geq t_0)$}         \\
                 &               & $0.90$ & $0.70$ & $0.50$& $0.30$   &  & $0.90$ & $0.70$ & $0.50$ & $0.30$ \\\hline
 20              &  $I_2$        & 0.864  & 0.833  & 0.760  & 0.605   &  & 0.079  & 0.065  & 0.055  & 0.043  \\
                 &  $I_1$        & 0.872  & 0.851  & 0.793  & 0.659   &  & 0.079  & 0.065  & 0.058  & 0.048  \\
\cline{1-11}
 40              &  $I_2$        & 0.887  & 0.872  & 0.846  & 0.777   &  & 0.057  & 0.048  & 0.045  & 0.041  \\
                 &  $I_1$        & 0.891  & 0.882  & 0.867  & 0.818   &  & 0.057  & 0.049  & 0.046  & 0.043  \\
\cline{1-11}
 60              &  $I_2$        & 0.892  & 0.888  & 0.870  & 0.822   &  & 0.046  & 0.040  & 0.037  & 0.036  \\
                 &  $I_1$        & 0.895  & 0.895  & 0.884  & 0.851   &  & 0.046  & 0.040  & 0.038  & 0.037  \\
\cline{1-11}
 80              &  $I_2$        & 0.892  & 0.888  & 0.878  & 0.845   &  & 0.040  & 0.035  & 0.033  & 0.032  \\
                 &  $I_1$        & 0.895  & 0.895  & 0.887  & 0.869   &  & 0.040  & 0.035  & 0.033  & 0.033  \\
\cline{1-11}
\end{tabular}
\end{table*}

The following are noted from Tables 4 and 5.

(1) As the sample size $n$ increases,  all of the coverage
proportions increase and are close to the nominal levels.

(2) The coverage proportions of $I_1$ are much better than that of
$I_2$.

(3) In 15/32 of the  cases, the average width of $I_2$ is slightly
shorter than that of $I_1$. In 17/32 of the cases, $I_2$ and $I_1$
have the same average width.

\section{Appendix}
\def\theequation{6. \arabic{equation}}
\setcounter{equation}{0}

\renewcommand{\proofname}{Proof of Lemma 3.2}
\begin{proof}
Put $\eta_n=o_p(1)$. By assumptions, for any $\vv>0$ and $\delta
>0$, there exist $M>0$, $b< b_H$ and $n_0>1$ such that for $n\geq
n_0$,
 $ P(|V_n|\geq M)\leq \delta$, $P( |h_n(b)|\geq \vv/M) \leq
P(|h(b)|\geq \vv/M) +\delta/2 \leq \delta$ and $P(|\eta_n|\geq
\vv)<\delta$ . It follows that for $n\geq n_0$,
\begin{eqnarray*}
P(|S_n|\geq 2\vv)&\leq & P(|\eta_n|\geq \vv) + P(|V_n h_n(b)|\geq \vv) \\
&\leq & \delta + P(|V_nh_n(b)| \geq \vv, |h_n(b)|\leq \vv/M ) +
\delta \\
&\leq & P(|V_n|\geq M)+2\delta\\
& \leq& 3\delta.
\end{eqnarray*}
\end{proof}
\renewcommand{\proofname}{Proof of Lemma 4.2}

\begin{proof}
The differences $W_{ni} - W_i$ in eq. (4.4) and (4.1)
can be expressed in terms of
\begin{eqnarray*}
 &&\gamma_i= \frac{\xi(Z_i)\delta_i}{\overline G_n(Z_i-)}-\frac{\xi(Z_i)\delta_i}{\overline G(Z_i)},\\
 &&\eta_i=\frac{{\overline \delta}_i}{{\overline H_n}(Z_i-)}\psi_n(Z_i)-\frac{{\overline \delta}_i}{{\overline H}(Z_i)}\psi(Z_i), \\
 && \nu_i= \int \psi_n(s)\frac{ I[Z_i\geq s]}{\overline H_n^2(s-)}dH_n^0(s)-\int \psi(s)\frac{ I[Z_i\geq s]}{\overline H^2(s)}dH^0(s),
\end{eqnarray*}
as
$$ W_{ni} - W_i = \gamma_i + \eta_i - \nu_i.$$
Applying an elementary inequality $(a+b+c)^2\leq 3(a^2+b^2+c^2)$, we
obtain
\begin{equation}
(W_{ni}-W_i )^2 =(\gamma_i + \eta_i - \nu_i)^2\leq 3(\gamma_i^2 +
\eta_i^2 + \nu_i^2). \label{6.1}
\end{equation}

The lemma will be proven by showing that the sample means of
$\gamma^2_i, \eta_i^2$ and $ \nu_i^2$ tend to zero in probability.
The proofs will be presented in (A), (B) and (C) below.

\vskip .12in

(A) \ The sample mean of $\gamma_i^2$ is $o_p(1)$.

Proof:  Let $G_n(x)$ be the K-M estimator defined in (2.3)
and $b< b_H$. Then as $n\to \infty$,
\begin{equation}
U_n=\sup_{s\leq b}\frac{|G_n(s-)-G(s)|}{\overline G_n(s)} = o_p(1),
\ \ V_n = \sup_{s\leq \max\{Z_i\}}\frac{|G_n(s-)-G(s)|}{\overline
G_n(s-)} = O_p(1). \label{6.2}
\end{equation}
See Zhou (1992).  To apply this result, we shall in the following
proof split the integrals into two intervals $[0, b]$ and $(b, b_H]$
accordingly.

For any $b<b_H$, using (2.2),
we have
\begin{eqnarray}
\frac{1}{n}\sum_{i=1}^n \gamma_i^2
 &=& \int \Big(\frac{\xi(s)}{\overline G_n(s-)}-\frac{\xi(s)}{\overline G(s)}\Big)^2 \,\rd H_n^1(s)\notag \\
 &\leq&  U_n^2 \int_0^b \frac{\xi^2(s)}{\overline G^2(s)}\,\rd H_n^1(s)+V_n^2 \int_b^{b_H}\frac{\xi^2(s)}{\overline G^2(s)} \,\rd H_n^1(s) \notag \\
 &=& o_p(1)O_p(1) + O_p(1)h_n(b) = o_p(1), \  \text{as} \ n\to \infty,
\label{6.3}
\end{eqnarray}
where $U_n^2$ and $V_n^2$ are given by (6.2) and
\begin{eqnarray}
 h_n(b)=\int_b^{b_H}\frac{\xi^2(s)}{\overline G^2(s)} \,\rd H_n^1(s).
\label{6.4}
\end{eqnarray}
Recall that $h_n(b)$ is explained in (3.12).
It was shown that $h_n(b)$ satisfies the conditions in Lemma 3.2.
The proof follows by invoking Lemma 3.2.

\vskip .15in

(B) \ The sample mean of $\eta_i^2$ is $o_p(1)$.

Proof: For $b<b_H$, define
$$ T_n(b,t]
 =\int_b^t\frac{\psi_n^2(s)}{\overline H_n^2(s-)} \,\rd H^0_n(s), \ \ S_n(b,t]
 =\int_b^t\frac{\psi^2(s)}{\overline H^2(s)}\,\rd H^0_n(s).$$
Observe that
$$\psi_n^2(s)=\Big(\int_{u\geq s} \xi(u) \,\rd F_n(u)\Big)^2 \leq\overline F_n(s-) \int_{u\geq s} \xi^2(u) \,\rd F_n(u),$$
and $F_n$ and $G_n$ have no common jumps. It follows that
\begin{eqnarray}
T_n(b, b_H]
 &\leq&\int_b^{b_H} \Big(\int_{u\geq s} \xi^2(u) \,\rd F_n(u)\Big) \overline F_n^2(s-)\frac{\,\rd G_n(s)}{\overline H_n^2(s-)}\notag \\
 &\leq&\int_b^{b_H} \Big( \int_{u\geq s} \xi^2(u) \,\rd F_n(u) \Big)\,\rd \Big(\frac{1}{\overline G_n(s)}\Big) \notag \notag \\
 &\leq&\lim_{s\to b_H}\frac{1}{\overline G_n(s)}\int_s^{b_H}\xi^2(u)\,\rd F_n(u)+\int_b^{b_H}\frac{\xi^2(s)}{\overline G_n(s-)}\,\rd F_n(s)\notag \\
 &\leq&2\int_b^{b_H}\frac{\xi^2(s)}{\overline G_n(s-)} \,\rd F_n(s) \notag \\
 &\leq&2\int_b^{b_H}\Big( \frac{1}{\overline G_n(s-)}-\frac{1}{\overline G(s)}+\frac{1}{\overline G(s)}\Big)^2 \xi^2(s)\,\rd H_n^1(s) \notag \\
 &=&o_p(1)+O_p(1) h_n(b). \label{6.5}
\end{eqnarray}
The first inequality follows from (2.6)
the second and the third from integration by parts, the fifth from
(2.5)
and the last equality from $(a+b)^2\leq 2(a^2+b^2)$, (6.3) and
(6.4).


By the same token, we conclude that
$S_n(b,b_H]=o_p(1)+O_p(1)h_n(b)$.

Write
$$\xi = \xi^+ - \xi^-, \ \psi_n(x) =  \int_{s\geq x} \xi^{+}(s) \,\rd F_n(s) - \int_{s\geq x} \xi^{-}(s) \,\rd F_n(s),$$
where $\xi^+$ and $\xi^-$ are the positive and negative part of
$\xi$. Define monotone functions:

$$\psi_n^{\pm}(x) = \int_{s\geq x} \xi^{\pm} (s) \,\rd F_n(s), \
\psi^{\pm}(x) =  \int_{s\geq x} \xi^{\pm}(s) \,\rd F(s).$$
$\psi_n^{\pm}(x)$ converges to $\psi^{\pm}(x)$ almost surely for $x
\in [0, b_H]$ as shown by Stute \& Wang (1993).  Furthermore, by
Lemma 4.1, the convergence is uniform on $[0, b_H]$.

From these we conclude the uniform convergence of $\psi_n$ to
$\psi$,
\begin{eqnarray}
 \sup_{0\leq x \leq b_H } |\psi_n(x)-\psi(x)| = o(1), \ \text{a.s.}.  \label{6.6}
\end{eqnarray}
Therefore, for $b< b_H$,
$$\sup_{s\leq b} \Big (\frac{\psi_n(s)} {\overline H_n(s-)}
-\frac{\psi(s)}{\overline H(s)}\Big)^2 \to 0,\ \text{a.s.}.$$

Applying $(a+b)^2\leq 2a^2 +2b^2$ and Lemma 3.2, we have
\begin{eqnarray}
\frac{1}{n}\sum_{i=1}^n \eta_i^2
 &=& \int_0^b+\int_b^{b_H}\Big(\frac{\psi_n(s)}{\overline H_n(s-)}-\frac{\psi(s)}{\overline H(s)}\Big)^2 \,\rd H^0_n(s)\notag \\
 &\leq& \int_0^b \Big(\frac{\psi_n(s)}{\overline H_n(s-)}-\frac{\psi(s)}{\overline H(s)}\Big)^2 \,\rd H^0_n(s) + 2T_n(b,b_H]+2S_n(b,b_H]\notag\\
 &=&o_p(1)+O_p(1) h_n(b) +O_p(1) h_n(b) =o_p(1). \label{6.7}
\end{eqnarray}

\vskip .15in

(C) \  The sample mean of $\nu_i^2$ is  $o_p(1)$.

Proof: Write, for $ 0 \leq a < t$,
$$B_n(a,t]=\int_a^t\frac{\psi_n(s)\,\rd H_n^0(s)}{\overline H_n^2(s-)}, \ \
D(a,t]=\int_a^t\frac{\psi(s)\,\rd H^0(s)}{\overline H^2(s)}.$$ Then,
for $b<b_H$, we have
\begin{eqnarray*}
\Delta_n^2
 &\equiv&\int_b^{b_H} B_n^2(b,t]\,\rd H_n(t)=\int_b^{b_H}B_n^2(b,t]\,\rd(-\overline H_n(t)) \\
 &\leq &\overline H_n(b) B_n^2(b,b] + 2 \int_b^{b_H}\overline H_n(t-) B_n(b,t] \frac{\psi_n(t)}{\overline H_n^2(t-)}\,\rd H_n^0(t)\\
 &\leq & 0+ 2 \Big( \int_b^{b_H} B_n^2(b,t] \,\rd H_n(t) \Big)^{1/2}\Big(\int_b^{b_H}\frac{\psi_n^2(t)}{\overline H_n^2(t-)} \,\rd H_n^0(t)\Big)^{1/2}\\
 &=& \Delta_n 2[T_n(b,b_H)]^{1/2}.
\end{eqnarray*}
The second term in the first inequality is obtained using the
Lebesgue-Stieltjes integration by parts.

Applying (6.5),
we get
$$\int_b^{b_H} B_n^2(b,t]\,\rd H_n(t)=\Delta_n^2\leq 4 T_n(b,b_H]= o_p(1)+O_p(1)h_n(b).$$
Similarly, for $S(b,b_H]=\E S_n(b,b_H] $, we have
$$ \int_b^{b_H} D^2(b,t]\,\rd H_n(t)
\leq 4 S(b,b_H] + o_p(1) = o_p(1)+ O_p(1)h_n(b).$$ Applying the
uniform convergence of $\psi_n$ to $\psi$ (see(6.6)),
we conclude that  for any $ b < b_H$, with probability 1, $ B_n(0,t]
\to D(0,t]$ uniformly on $[0, b]$. It follows that $B_n(0,b]\to
D(0,b]$, and
\begin{eqnarray*}
\frac{1}{n}\sum_{i=1}^n \nu_i^2
&=&\int_0^b + \int_b^{b_H} \big ( B_n(0,t]-D(0,t]\big)^2 \,\rd H_n(t) \notag\\
&=&o_p(1) +\int_b^{b_H} \big( B_n(0,b]-D(0,b] + B_n(b, t]-D(b,t]\big )^2 \,\rd H_n(t)\notag\\
&\leq&o_p(1) + 4 \int_b^{b_H} B_n^2(b,t] \,\rd H_n(t) + 4\int_b^{b_H}D^2(b,t) \,\rd H_n(t) \notag\\
&=& o_p(1) + O_p(1)h_n(b) + O_p(1)h_n(b) =o_p(1).
\end{eqnarray*}

Now, we  prove result (2) of the lemma. Introduce
$A_{ni}=W_i{\overline H}_n(Z_i-) {\overline H}(Z_i)$. From (6.1)
we get
$$ (V_{ni}-A_{ni})^2 \leq 3(\gamma_i^2 + \eta_i^2 + \nu_i^2){\overline H}^2_n(Z_i-) {\overline H}^2(Z_i).$$
Similar to (6.3) and (6.7),
we have
\begin{eqnarray*}
& &\frac{1}{n}\sum_{i=1}^n \gamma_i^2 {\overline
H}^2_n(Z_i-){\overline
 H}^2(Z_i)\\
& =&\int_0^{b_H}\Big(\frac{\xi(s)}{\overline
G_n(s-)}-\frac{\xi(s)}{\overline G(s)}\Big)^2{\overline
H}^2_n(s-){\overline H}^2(s)\,\rd H_n^1(s)
 =o(1),\ \text{a.s.}, \\
 & &\frac{1}{n}\sum_{i=1}^n \eta_i^2{\overline H}^2_n(Z_i-) {\overline
 H}^2(Z_i)\\
& =&\int_0^{b_H} \Big (\frac{\psi_n(s)} {\overline
H_n(s-)}-\frac{\psi(s)}{\overline H(s)}\Big)^2 {\overline
H}^2_n(s-){\overline H}^2(s)\,\rd H^0_n(s)
 =o(1),\ \text{a.s.}.
\end{eqnarray*}

Since for any $b<b_H$, with probability 1,
$$ |B_n(0,t]-D(0,t]|{\overline H_n(t-)} {\overline H}(t) \to 0$$
uniformly on $[0, b]$, and
$$\sup_{t\leq b_H}|B_n(0,t]- D(0,t]|{\overline H_n(t-)}{\overline H}(t)$$
is bounded by some constant, it follows that
\begin{eqnarray*}
& & \frac{1}{n}\sum_{i=1}^n \nu_i^2{\overline H}^2_n(Z_i-)
{\overline
H}^2(Z_i) \\
& =&\int_0^{b_H} \big [ (B_n(0,t]-D(0,t]){\overline H}_n(t-)
{\overline H}(t) \big ]^2 \,\rd H_n(t)
 =o(1),\ \text{a.s.}.
\end{eqnarray*}
Now we get
$$ \frac{1}{n}\sum_{i=1}^{n}(V_{ni}-A_{ni})^2
 \leq 3\frac{1}{n}\sum\limits_{i=1}^{n}(\gamma_i^2 + \eta_i^2+\nu_i^2){\overline H}^2_n(Z_i-) {\overline H}^2(Z_i)
 =o(1),\ \text{a.s.}.$$
At last, we have
\begin{eqnarray}
\frac{1}{n}\sum\limits_{i=1}^{n}(V_{ni} -V_i)^2
 &\leq& \frac{2}{n}\sum\limits_{i=1}^{n}(V_{ni} -A_{ni})^2+\frac{2}{n}\sum\limits_{i=1}^{n}(A_{ni}-V_i)^2 \notag\\
 &=&o(1)+\frac{2}{n}\sum\limits_{i=1}^{n}(H(Z_i)-H_n(Z_i-))^2W_i^2\overline H^2(Z_i)\notag\\
 &=&o(1),\ \text{a.s.}. \notag
\end{eqnarray}

\end{proof}

\renewcommand{\proofname}{Proof of Lemma 4.3}

\begin{proof}
Since $W_i$ are i.i.d. random variables with zero mean and finite
variance $\sig^2$, hence $\max\limits_{1\leq i\leq n} |W_i|=o_p(
\sqrt{n})$. It follows from Lemma 4.2 that
\begin{eqnarray}
\max\limits_{1\leq i\leq n}|W_{ni}|
 &\leq& \Big(\max\limits_{1\leq i\leq n} |W_{ni}-W_i|^2\Big)^{1/2}+\max\limits_{1\leq i\leq n} |W_i| \notag \\
 &\leq& \sqrt{n}\Big(\frac{1}{n}\sum_{i=1}^n (W_{ni}-W_i)^2\Big)^{1/2} +o_p(\sqrt{n})\notag \\
 &=& o_p( \sqrt{n}). \label{6.8}
\end{eqnarray}

Note that $W_{ni}^2 $ is bounded by
$$W^2_{i} +(W_i-W_{ni})^2 -2|W_i(W_i-W_{ni})|\leq W_{ni}^2
 \leq W_i^2+(W_{ni}-W_i)^2+2|W_i(W_i-W_{ni})|.$$
By Lemma 4.2, we get
$$\lim_{n\to\infty}\frac{1}{n}\sum_{i=1}^{n}W^2_{ni}=\sig^2 + o_p(1).$$

The last result follows from  Theorem 3.1, $\xi(x)=g(x,\th_0)$ and
(4.6).


\end{proof}

\end{document}